\documentclass{article}
\usepackage{amsmath,amssymb,amsthm}
\usepackage{authblk}
\usepackage[all]{xy}
\usepackage{color}
\usepackage{mathtools}
\theoremstyle{definition}
\newtheorem{thm}{Theorem}[section]
\newtheorem{defi}[thm]{Definition}
\newtheorem{lemm}[thm]{Lemma}
\newtheorem{prop}[thm]{Proposition}
\newtheorem{coro}[thm]{Corollary}

\newtheorem{rema}[thm]{Remark}

\numberwithin{equation}{section}
\allowdisplaybreaks
\SelectTips{cm}{12}
\objectmargin={1pt}
\title{Strong minuscule elements \\ in the finite Weyl groups}
\author{Yuki Motegi\thanks{Graduate School of Pure and Applied Sciences, University of Tsukuba, 1-1-1 Tennodai, Tsukuba, Ibaraki 305-8571, Japan (e-mail : y-motegi@math.tsukuba.ac.jp)}}
\date{}
\begin{document}
\maketitle

\begin{abstract}
We introduce the notion of a strong minuscule element, which is a dominant minuscule element $w$ in the Weyl group for which there exists  a unique (dominant) integral weight $\Lambda$ such that $w$ is $\Lambda$-minuscule. Then we prove that the dominant integral weight associated to a strong minuscule element is the fundamental weight corresponding to a short simple root (in this paper, all simple roots in the simply-laced cases are treated as short roots). In addition, we enumerate the strong minuscule elements explicitly, and then as an application of this enumeration, determine the dimension of certain Demazure modules in the finite-dimensional irreducible modules whose highest weights are minuscule weights.
\end{abstract}

\section{Introduction.}\label{sec:1}

The notion of (dominant) minuscule elements in the Weyl group was introduced by Peterson \cite{Carrell} in order to study the number of reduced expressions for an element in the Weyl group. For the definition of a (dominant) minuscule element, see Definition \ref{defi:3.1} below.
In this paper, we study the following special class of dominant minuscule elements in the Weyl group for the finite-dimensional simple Lie algebras $\mathfrak{g}$. A dominant minuscule element $w$ in the Weyl group of $\mathfrak{g}$ is called a strong minuscule element if there exists unique dominant integral weight $\Lambda$ (which we denote by $\Lambda_{w}$) such that $w$ is $\Lambda$-minuscule. We denote by $\mathcal{SM}$ the set of a strong minuscule elements. We prove that the dominant integral weight $\Lambda_{w}$ associated to a strong minuscule element $w$ is the fundamental weight corresponding to a short simple root (see Proposition \ref{lemm:6.2} and Appendix; in this paper, all simple roots in the simply-laced cases are treated as short roots). Let $\{\alpha_{i}\}_{i \in I}$ be the set of simple roots for $\mathfrak{g}$, with $I = \{1, 2, \dots , n\}$.
\begin{prop}[see Corollary \ref{rema:6.4} and Appendix]\label{intro1}
It holds that
\[
\mathcal{SM} = \bigsqcup_{i \in K} \mathcal{SM}_{i},
\]
where $K \coloneqq \{i \in I \mid \alpha_{i}\ \text{is short}\}$, and $\mathcal{SM}_{i} = \{w \in \mathcal{SM} \mid \Lambda_{w} = \Lambda_{i}\}$ for $i \in K$.
\end{prop}
In addition, we enumerate the strong minuscule elements in $\mathcal{SM}_{i}$ for $i \in K$ explicitly (see Proposition \ref{prop:enumerate} and Appendix). As an application of this result, we obtain the following dimension formula of certain Demazure modules in finite-dimensional irreducible $\mathfrak{g}$-modules. Let and fix $i \in K$, and set $J_{i} \coloneqq \{s_{1}, \dots , s_{n}\} \backslash \{s_{i}\}$, where $s_{i} \in W$ is the simple reflection in the simple root $\alpha_{i}$. We set $\overline{v_{i}} \coloneqq w_{0}v_{i}w_{J_{i},0}$, where $v_{i} \in W$ is defined in Section \ref{sec:6} and Appendix, and where $w_{0}$ (resp., $w_{J_{i}, 0}$) is the longest element of $W$ (resp., of the parabolic subgroup $W_{J_{i}}$ of $W$ generated by $J_{i}$).
\begin{thm}[$=$ Theorem \ref{thm:5}]
Let $i \in K$ be such that $\Lambda_{i}$ is a minuscule weight. It hold that
\[
\text{dim}\, E_{\overline{v_{i}}} (\Lambda_{i}) =
\begin{cases}
\binom{n-1}{i-1}\ &(1 \leq i \leq n\ \text{in type}\ \text{A}_{n}), \\
2^{n-1}\ &(i = 1\ \text{in type}\ \text{B}_{n}), \\
n+1\ &(i=n\ \text{in type}\ \text{C}_{n}), \\
2^{n-2}-1\ &(i=1,2\ \text{in type}\ \text{D}_{n}), \\
n\ &(i=n\ \text{in type}\ \text{D}_{n}), \\
16\ &(i=1,5\ \text{in type}\ \text{E}_{6}), \\
43\ &(i=6\ \text{in type}\ \text{E}_{7}).
\end{cases}
\]
where $E_{\overline{v_{i}}} (\Lambda_{i}) \coloneqq U(\mathfrak{n}_{+}) L(\Lambda_{i})_{\overline{v_{i}} (\Lambda_{i})}$ is the Demazure module of the lowest weight $\overline{v_{i}} (\Lambda_{i})$ in the finite-dimensional irreducible $\mathfrak{g}$-module $L(\Lambda_{i})$ of highest weight $\Lambda_{i}$.
\end{thm}

\subsubsection*{Acknowledgements :}
The author would like to thank Daisuke Sagaki, who is his supervisor, for useful discussions. He also thank Ryo Kawai and Masato Tada for program implementation.

\section{Preliminaries.}\label{sec:2}
Let $\mathbb{N}$ denote the set of nonnegative integers. Throughout this paper, except for Appendix, $\mathfrak{g}$ is the finite-dimensional classical simple Lie algebra of type $\text{A}_{n}$, $\text{B}_{n}$, $\text{C}_{n}$, or $\text{D}_{n}$ over $\mathbb{C}$; the Dynkin diagram for $\mathfrak{g}$ is as follows:
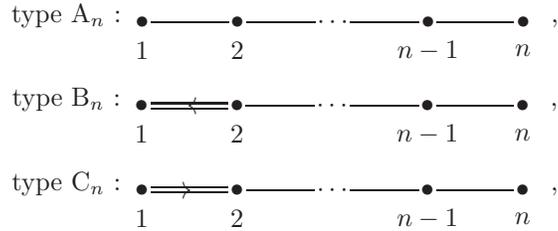
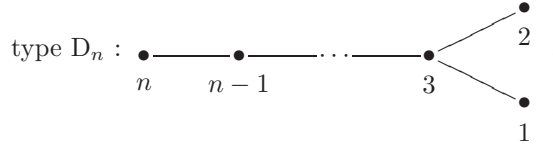
\begin{center}
type $\text{A}_{n}$ : 
\xygraph{
    \bullet ([]!{+(0,-.3)} {1}) - [r]
    \bullet ([]!{+(0,-.3)} {2}) - [r] \cdots - [r]
    \bullet ([]!{+(0,-.3)} {n - 1}) - [r]
    \bullet ([]!{+(0,-.3)} {n})}\ \ ,

\vspace{3mm}
type $\text{B}_{n}$ :
\xygraph{!~:{@{=}|@{<}}
    \bullet ([]!{+(0,-.3)} {1}) : [r]
    \bullet ([]!{+(0,-.3)} {2}) - [r] \cdots - [r]
    \bullet ([]!{+(0,-.3)} {n - 1}) - [r]
    \bullet ([]!{+(0,-.3)} {n})}\ \ ,

\vspace{3mm}
type $\text{C}_{n}$ :
\xygraph{!~:{@{=}|@{>}}
    \bullet ([]!{+(0,-.3)} {1}) : [r]
    \bullet ([]!{+(0,-.3)} {2}) - [r] \cdots - [r]
    \bullet ([]!{+(0,-.3)} {n - 1}) - [r]
    \bullet ([]!{+(0,-.3)} {n})}\ \ ,

\vspace{3mm}
type $\text{D}_{n}$ :
\xygraph{
    \bullet ([]!{+(0,-.3)} {n}) - [r]
    \bullet ([]!{+(0,-.3)} {n-1}) - [r] \cdots - [r]
    \bullet ([]!{+(0,-.3)} {3}) (
        - []!{+(1,.5)} \bullet ([]!{+(0,-.3)} {2}),
        - []!{+(1,-.5)} \bullet ([]!{+(0,-.3)} {1})
)}\ \ .
\end{center}
Let $(a_{ij})_{i,j \in I}$ be the Cartan matrix of $\mathfrak{g}$, where $I = \{1, 2, \dots , n\}$. Let $\mathfrak{h}$ be the Cartan subalgebra of $\mathfrak{g}$, and set $\mathfrak{h}^{\ast} \coloneqq \text{Hom}_{\mathbb{C}} (\mathfrak{h}, \mathbb{C})$.
We denote by $\langle \cdot , \cdot \rangle : \mathfrak{h}^{\ast} \times \mathfrak{h} \to \mathbb{C}$ the standard pairing.
Denote by  $\Pi = \{\alpha_{i} \mid i \in I\}$ (resp., $\Pi^{\vee} = \{\alpha_{i}^{\vee} \mid i \in I\}$) the set of simple roots (resp., simple coroots); note that $\langle \alpha_{j}, \alpha_{i}^{\vee} \rangle = a_{ij}$.

Let $P = \bigoplus_{i \in I} \mathbb{Z}\Lambda_{i}$ (resp., $P^{+} = \sum_{i \in I} \mathbb{Z}_{\geq 0}\Lambda_{i}$) be the set of integral weights (resp., dominant integral weights), where $\Lambda_{i}$ is the fundamental weight. We denote by $W = \langle s_{i} \mid i \in I \rangle \subset GL(\mathfrak{h}^{\ast})$ the Weyl group of $\mathfrak{g}$, where $s_{i}$ is the simple reflection in $\alpha_{i}$, and denote by $\ell : W \to \mathbb{Z}_{\geq 0}$ the length function on $W$. Denote by $\Phi$ (resp., $\Phi_{+}$) the set of roots (resp., positive roots) for $\mathfrak{g}$. For $\beta \in \Phi$, $\beta^{\vee}$ denotes the coroot of $\beta$.

Let $K$ be the subset of $I = \{1,2, \dots , n\}$ given as follows:
\begin{equation}\label{eq:22}
K \coloneqq
\begin{cases}
\ \ \ I\ &\text{if}\ \text{type}\ \text{A}_{n}\ \text{or}\ \text{D}_{n}, \\
\ \ \{1\}\ &\text{if}\ \text{type}\ \text{B}_{n}, \\
\ I \backslash \{1\}\ &\text{if}\ \text{type}\ \text{C}_{n}.
\end{cases}
\end{equation}
Namely, the set $K$ is identical to $I$ if $\mathfrak{g}$ is of type $\text{A}_{n}$ or $\text{D}_{n}$, and to $\{i \in I \mid \alpha_{i}\ \text{is a short simple root}\}$ if $\mathfrak{g}$ is of type $\text{B}_{n}$ or $\text{C}_{n}$. For $i \in I$, we set
\begin{align*}
\text{adj} (i) \coloneqq \{j \in I \mid a_{ij} \neq 0, 2\}&,\ \ \text{adj}_{s} (i) \coloneqq \{j \in \text{adj} (i) \mid a_{ij} = -1\}, \\
\text{adj}_{\ell} (i) \coloneqq \text{adj} (i) \backslash \text{adj}_{s} (i) &= \{j \in \text{adj} (i) \mid a_{ij} = -2\}.
\end{align*}

\section{Minuscule elements in the Weyl group.}\label{sec:3}
\begin{defi}[see \cite{Carrell}, \cite{Proctor 1}]\label{defi:3.1}
Let $\Lambda \in P$. A Weyl group element $w \in W$ is said to be $\Lambda${\it -minuscule} if there exists a reduced expression $w = s_{i_{1}} \cdots s_{i_{r}}$ such that
\begin{equation}\label{eq:3.1}
\langle s_{i_{p+1}} \cdots s_{i_{r}} (\Lambda), \alpha_{i_{p}}^{\vee} \rangle = 1\ \text{for\ all}\ 1 \leq p \leq r.
\end{equation}
If $w \in W$ is $\Lambda$-minuscule for some integral weight $\Lambda \in P$ (resp., dominant integral weight $\Lambda \in P^{+}$), then we say that $w$ is {\it minuscule} (resp., {\it dominant\ minuscule}). The set of minuscule  (resp., dominant minuscule) elements in $W$ is denoted by $\mathcal{M}$ (resp., $\mathcal{M}^{+}$).
\end{defi}

\begin{rema}\label{rema:3.2}
Let $\Lambda \in P$, and $w \in W$. If condition (\ref{eq:3.1}) holds for some reduced expression of $w$, then it holds for every reduced expression of $w$.
Hence the definition of a $\Lambda$-minuscule element is independent of the choice of a reduced expression.
\end{rema}

\section{Strong minuscule elements.}\label{sec:4}
\begin{defi}\label{defi:4.1}
A dominant minuscule element $w \in W$ is said to be {\it strong\ minuscule} if there exists a unique dominant integral weight $\Lambda \in P^{+}$ (which we denote by $\Lambda_{w}$) such that $w$ is $\Lambda$-minuscule. The set of strong minuscule elements in $W$ is denoted by $\mathcal{SM}$.
\end{defi}

\begin{prop}\label{prop:4.2}
Let $w \in \mathcal{SM}$, and $w = s_{i_{1}} \cdots s_{i_{r}}$ be a reduced expression of $w$. Then, $\# \{1 \leq p \leq r \mid i_{p} = i\} \geq 1$ for each $i \in I$. Namely, each of the simple reflections appears at least once in each reduced expression of $w$.
\end{prop}

\begin{proof}
Suppose, for a contradiction, that $s_{j}$ does not appear in the reduced expression $w = s_{i_{1}} \cdots s_{i_{r}}$ for some $j \in I$. In this case, since $s_{i_{p+1}} \cdots s_{i_{r}} (\Lambda_{j}) = \Lambda_{j}$ and $\langle \Lambda_{j}, \alpha_{i_{p}}^{\vee} \rangle = 0$ for all $1 \leq p \leq r$, we see that $w$ is also $(\Lambda_{w} + \Lambda_{j})$-minuscule. Because $\Lambda_{w} + \Lambda_{j} \in P^{+}$, this contradicts the assumption that $w \in \mathcal{SM}$.
\end{proof}

\section{Main results.}\label{sec:5}
\begin{prop}[will be proved in \S\ref{sec:6}]\label{main result 1}
It holds that
\[
\mathcal{SM} = \bigsqcup_{i \in K} \mathcal{SM}_{i},
\]
where $\mathcal{SM}_{i} \coloneqq \{w \in \mathcal{SM} \mid \Lambda_{w} = \Lambda_{i}\}$.
\end{prop}

\begin{prop}[will be proved in \S\ref{sec:7}]\label{main result 2}
It hold that \\
(i) If $\mathfrak{g}$ is of type $\text{A}_{n}$, then $\# \mathcal{SM}_{i} = \binom{n-1}{i-1}$ for $1 \leq i \leq n$; \\
(ii) If $\mathfrak{g}$ is of type $\text{B}_{n}$, then $\# \mathcal{SM}_{1} = 2^{n-1}$; \\
(iii) If $\mathfrak{g}$ is of type $\text{C}_{n}$, then $\# \mathcal{SM}_{i} = \binom{n-1}{i-2}$ for $2 \leq i \leq n-1$, and $\# \mathcal{SM}_{n} = n$; \\
(iv) If $\mathfrak{g}$ is of type $\text{D}_{n}$, then $\# \mathcal{SM}_{1} = \# \mathcal{SM}_{2} = 2^{n-2}-1$, $\# \mathcal{SM}_{i} = \binom{n-2}{i-3}$ for $3 \leq i \leq n-1$, and $\# \mathcal{SM}_{n} = n-1$.
\end{prop}

We denote by $\leq$ the Bruhat order on $W$. For $u, w \in W$, we set $[u, w] \coloneqq \{v \in W \mid u \leq v \leq w\}$. Denote by $w_{0}$ the longest element in $W$.
For $i \in I$, let $W_{J_{i}}$ be the (parabolic) subgroup of $W$ generated by $J_{i} \coloneqq \{s_{1}, \dots , s_{n}\} \backslash \{s_{i}\}$, and $W^{J_{i}} (\subset W)$ the set of minimal-length coset representatives of cosets in $W / W_{J_{i}}$. Let $w_{0}^{J_{i}} \in W^{J_{i}}$ be such that $w_{0}^{J_{i}} \in w_{0}W_{J_{i}}$. For $u, w \in W^{J_{i}}$, we set $[u, w]^{J_{i}} \coloneqq [u, w] \cap W^{J_{i}}$.

\begin{prop}[will be proved in \S\ref{sec:8}]
It hold that \\
(i') If $\mathfrak{g}$ is of type $\text{A}_{n}$, then $\mathcal{SM}_{i} = [v_{i}, w_{0}^{J_{i}}]^{J_{i}}$ for $1 \leq i \leq n$; \\
(ii') If $\mathfrak{g}$ is of type $\text{B}_{n}$, then $\mathcal{SM}_{1} = [v_{1}, w_{0}^{J_{1}}]^{J_{1}}$; \\
(iii') If $\mathfrak{g}$ is of type $\text{C}_{n}$, then $\mathcal{SM}_{n} = [v_{n}, w_{0}^{J_{n}}]^{J_{n}} \backslash \{w_{0}^{J_{n}}\}$; \\
(iv') If $\mathfrak{g}$ is of type $\text{D}_{n}$, then $\mathcal{SM}_{1} = [v_{1}, w_{0}^{J_{1}}]^{J_{1}}$, $\mathcal{SM}_{2} = [v_{2}, w_{0}^{J_{2}}]^{J_{2}}$, and $\mathcal{SM}_{n} = [v_{n}, w_{0}^{J_{n}}]^{J_{n}} \backslash \{w_{0}^{J_{n}}\}$.
\end{prop}

\section{Properties of strong minuscule elements.}\label{sec:6}
\begin{lemm}[{\cite[Proposition 2.5]{Stembridge 1}}]\label{lemm:6.1}
Let $w \in \mathcal{M}^{+}$, and fix a reduced expression $w = s_{i_{1}} \cdots s_{i_{r}}$ of $w$. Fix $i \in I$, and set $a \coloneqq \text{max}\{1 \leq p \leq r \mid i_{p} = i\}$.
Then,
\begin{align}
\# \{a + 1 \leq p \leq r \mid i_{p} \in \text{adj}_{s} (i)\} \leq 1, \label{align:6.1} \\
\# \{a + 1 \leq p \leq r \mid i_{p} \in \text{adj}_{\ell} (i)\} = 0. \label{align:6.2}
\end{align}
\end{lemm}

\begin{rema}\label{rema:6.3}
Let $w \in \mathcal{M}^{+}$, and $w = s_{i_{1}} \cdots s_{i_{r}}$ be a reduced expression of $w$. We claim that if $i_{r} \in I \backslash K$, then $w \notin \mathcal{SM}$. Indeed, suppose, for a contradiction, that $w \in \mathcal{SM}$. By Proposition \ref{prop:4.2}, there exists $1 \leq p \leq r-1$ such that $i_{r} \in \text{adj}_{\ell} (i_{p})$. This contradicts (\ref{align:6.2}).
\end{rema}

Recall that $K$ is as (\ref{eq:22}). For $i \in K$, we define $v_{i} \in W$ as follows: \\
(a) If $\mathfrak{g}$ is of type $\text{A}_{n}$, then $v_{i} \coloneqq s_{n}s_{n-1} \cdots s_{i+1}s_{1}s_{2} \cdots s_{i-1}s_{i}$ for $i \in K = I$; \\
(b) If $\mathfrak{g}$ is of type $\text{B}_{n}$, then $v_{1} \coloneqq s_{n} s_{n-1} \cdots s_{2} s_{1}$; \\
(c) If $\mathfrak{g}$ is of type $\text{C}_{n}$, then $v_{i} \coloneqq s_{n} s_{n-1} \cdots s_{i+1} s_{1}s_{2} \cdots s_{i-1}s_{i}$ for $i \in K = I \backslash \{1\}$; \\
(d) If $\mathfrak{g}$ is of type $\text{D}_{n}$, then $v_{1} \coloneqq s_{2}s_{n}s_{n-1} \cdots s_{3}s_{1}, v_{2} \coloneqq s_{1}s_{n}s_{n-1} \cdots s_{3}s_{2}$, \text{and} \\
$v_{i} \coloneqq s_{n}s_{n-1} \cdots s_{i+1}s_{1}s_{2}s_{3} \cdots s_{i-1}s_{i}$ for $i \in K \backslash \{1,2\} = I \backslash \{1,2\}$. \\
In all cases, it holds that $\ell(v_{i}) = n$.

\begin{lemm}\label{lemm:6.2}
Let $w \in \mathcal{M}$, and let $w = s_{i_{1}} \cdots s_{i_{r}}$ be a reduced expression of $w$. Set $k \coloneqq i_{r} \in I$. Then, $w$ is a strong minuscule element if and only if $k \in K$ and there exists $u \in W$ such that $w = uv_{k}$ and $\ell (w) = \ell(u)+n$. Moreover, it holds that $\Lambda_{w} = \Lambda_{k}$ in this case.
\end{lemm}

\begin{proof}
We give a proof only for the cases of type $\text{A}_{n}$, $\text{B}_{n}$, or $\text{C}_{n}$; the proof for the case of type $\text{D}_{n}$ is similar. Assume that $w \in \mathcal{SM}$; in particular, $w \in \mathcal{M}^{+}$. It follows from Remark \ref{rema:6.3} that $k \in K$. First, we show by (descending) induction on $1 \leq p \leq k$ (starting from $p = k$) that $w$ has a reduced expression of the form
\begin{equation}\label{eq:6.4}
w = \cdots s_{p} s_{p+1} \cdots s_{k-1} s_{k} .
\end{equation}
If $p = k$, then the assertion is obvious by assumption. Assume that $1 < p \leq k$; by the induction hypothesis, we have a reduced expression for $w$ of the form:
\begin{equation}\label{eq:6.5}
w = \cdots s_{p} s_{p+1} \cdots s_{k-1} s_{k}.
\end{equation}
By Proposition \ref{prop:4.2}, $s_{p-1}$ appears in this reduced expression. Let us take the right-most one:
\begin{equation}\label{eq:6.6}
w = \cdots s_{p-1} \underbrace{\cdots}_{(\ast)} s_{p} s_{p+1} \cdots s_{k-1} s_{k};
\end{equation}
there is no $s_{p-1}$ in $(\ast)$. Also, by (\ref{align:6.1}), neither $s_{p}$ nor $s_{p-2}$ appears in $(\ast)$, which implies that every simple reflection in $(\ast)$ commutes with $s_{p-1}$. Hence, we get a reduced expression for $w$ of the form:
\begin{equation}\label{eq:6.7}
w = \cdots s_{p-1} s_{p} s_{p+1} \cdots s_{k-1} s_{k},
\end{equation}
as desired. In particular, we obtain a reduced expression of the form
\begin{equation}\label{eq:6.8}
w = \cdots s_{1} s_{2} \cdots s_{k-1} s_{k}.
\end{equation}
Similarly, we can show by induction on $k \leq q \leq n$ that $w$ has a reduced expression of the form:
\begin{equation*}
w = \cdots s_{q} \cdots s_{k+2} s_{k+1} s_{1} s_{2} \cdots s_{k-1} s_{k}.
\end{equation*}
In particular, we obtain a reduced expression of the form
\begin{equation}\label{eq:6.9}
w = \underbrace{\cdots}_{\eqqcolon u} \underbrace{s_{n} s_{n-1} \cdots s_{k+2} s_{k+1} s_{1} s_{2} \cdots s_{k-1} s_{k}}_{= v_{k}}.
\end{equation}
If we set $u \coloneqq wv_{k}^{-1}$, then we have $w = uv_{k}$ with $\ell (w)$ = $\ell (u) + n$, as desired.

Conversely, assume that ($w \in \mathcal{M}$, and) there exists $u \in W$ such that $w = uv_{k}$ with $\ell(w) = \ell(u) + n$; note that $w$ has a reduced expression of the form (\ref{eq:6.9}).
Let $\Lambda \in P$ be such that $w$ is $\Lambda$-minuscule, and write it $\Lambda$ as: $\Lambda = \sum_{i=1}^{n} c_{i}\Lambda_{i}$ with $c_{i} \in \mathbb{Z}$. Since $\langle \Lambda , \alpha_{k}^{\vee} \rangle = 1$ by the assumption that $w$ is $\Lambda$-minuscule (see also Remark \ref{rema:3.2}), we get $c_{k} = 1$. Also, we see that $\langle \Lambda - \alpha_{k} , \alpha_{k-1}^{\vee} \rangle = 1$ and $k \in \text{adj}_{s} (k-1)$, which implies that $c_{k-1} = 0$. Repeating this argument, we get $c_{k-1} = c_{k-2} = \cdots = c_{1} = 0$. Similarly, we see that $\langle \Lambda - \alpha_{k} - \alpha_{k-1} - \cdots - \alpha_{1} , \alpha_{k+1}^{\vee} \rangle = 1$ and $k \in \text{adj}_{s}(k+1)$, which implies that $c_{k+1} = 0$. Repeating this argument, we get $c_{k+2} = c_{k+3} = \cdots = c_{n} = 0$. Therefore, we conclude that $\Lambda = \Lambda_{k} \in P^{+}$; in paticular, $w$ is dominant minuscule. Furthermore, the argument above shows the uniqueness of $\Lambda \in P^{+}$ such that $w$ is $\Lambda$-minuscule. Thus we have proved Lemma \ref{lemm:6.2}.
\end{proof}

For each $i \in K$, we set $\mathcal{SM}_{i} \coloneqq \{w \in \mathcal{SM} \mid \Lambda_{w} = \Lambda_{i}\}$. The next corollary follows immediately from Lemma \ref{lemm:6.2} and the definition of a strong minuscule element.

\begin{coro}\label{rema:6.4}
It holds that
\begin{equation}\label{eq:6.10}
\mathcal{SM} = \bigsqcup_{i \in K} \mathcal{SM}_{i}.
\end{equation}
\end{coro}

\begin{lemm}\label{lemm:6.5}
Let $w \in \mathcal{SM}_{i}$, and $w = s_{i_{1}} \cdots s_{i_{r}}$ be a reduced expression of $w$; recall that $i_{r} = i$. For each $1 \leq p \leq r-1$, we set $u_{p} \coloneqq \# \{p+1 \leq a \leq r \mid i_{a} \in \text{adj}_{s} (i_{p})\}$.
Then,
\begin{align}
u_{p} \in 2\mathbb{N} \ \ \ \ \ \ &\text{if}\ \ i_{p} = i, \label{align:6.11} \\
u_{p} \in 2\mathbb{N}+1 \ \ \ &\text{if}\ \ i_{p} \neq i. \label{align:6.12}
\end{align}
\end{lemm}

\begin{proof}
By (\ref{eq:3.1}), we have $\langle \Lambda_{i} - \alpha_{i_{r}} - \cdots - \alpha_{i_{p+1}}, \alpha_{i_{p}}^{\vee} \rangle = 1$ for all $1 \leq p \leq r$, and hence $\delta_{i, i_{p}} - a_{i_{p}, i_{r}} - \cdots - a_{i_{p}, i_{p+1}} = 1$. Now, we set $t_{p} \coloneqq \# \{p+1 \leq a \leq r \mid i_{a} \in \text{adj}_{\ell} (i_{p})\}$ and $q_{p} \coloneqq \# \{p+1 \leq a \leq r \mid i_{a} = i_{p}\}$. If $i_{p} = i$, then $1 - u_{p} -2t_{p} + 2q_{p} = 1$. Therefore, we obtain $u_{p} = 2(q_{p} - t_{p}) \in 2 \mathbb{N}$. If $i_{p} \neq i$, then $- u_{p} -2t_{p} + 2q_{p} = 1$. Hence we have $u_{p} = 2(q_{p} - t_{p}) -1 \in 2\mathbb{N}+1$. Thus we have proved the lemma.
\end{proof}

\section{Enumeration for the strong minuscule elements.}\label{sec:7}
Recall that the Weyl group $W$ of $\mathfrak{g}$ is generated by $S \coloneqq \{s_{1}, \dots , s_{n}\}$. For $J \subset S$, let $W_{J}$ be the (parabolic) subgroup of $W$ generated by $J$. Let $W^{J} \cong W / W_{J}$ be the set of minimal-length coset representatives for cosets in $W / W_{J}$ (see \cite[Corollary 2.4.5]{BB}).
Put $J_{i} \coloneqq S \backslash \{s_{i}\} \subset S$ for $i \in I$.
For $j \in I$, we define $w_{j} \in W$ as follows: \\
(a') If $\mathfrak{g}$ is of type $\text{A}_{n}$, then $w_{j} \coloneqq s_{1} s_{2} \cdots s_{j-1} s_{j}$ for $j \in I$; \\
(b') If $\mathfrak{g}$ is of type $\text{B}_{n}$, then $w_{j} \coloneqq s_{n}s_{n-1} \cdots s_{2}s_{1}s_{2} \cdots s_{j-1}s_{j}$ for $j \in I$; \\
(c') If $\mathfrak{g}$ is of type $\text{C}_{n}$, then $w_{j} \coloneqq s_{n}s_{n-1} \cdots s_{2}s_{1}s_{2} \cdots s_{j-1}s_{j}$ for $j \in I$; \\
(d') If $\mathfrak{g}$ is of type $\text{D}_{n}$, then $w_{1} \coloneqq s_{n}s_{n-1} \cdots s_{4}s_{3}s_{1}$, $w_{2} \coloneqq s_{n}s_{n-1} \cdots s_{4}s_{3}s_{2}$, \text{and} $w_{j} \coloneqq s_{n}s_{n-1} \cdots s_{3}s_{1}s_{2}s_{3} \cdots s_{j-1}s_{j}$ for $j \in I \backslash \{1,2\}$. \\
For $j \in I$ and $0 \leq l \leq \ell(w_{j})$, define $w_{j}(l)$ to be the product of $l$ simple reflections from the right in the expression of $w_{j}$ above, except for the case that $\mathfrak{g}$ is of type $\text{D}_{n}$, $j \in I \backslash \{1,2\}$, and $l = j-1$. When $\mathfrak{g}$ is of type $\text{D}_{n}$, and $j \in I \backslash \{1,2\}$, the element $w_{j}(j-1)$ represents both $s_{1}s_{3} \cdots s_{j}$ and $s_{2}s_{3} \cdots s_{j}$; for example, the sentence ``a proposition holds for $w_{j}(j-1)$'' means that the proposition holds for both $s_{1}s_{3} \cdots s_{j}$ and $s_{2}s_{3} \cdots s_{j}$.

\begin{prop}[{\cite[Theorems 2 and 6]{Stumbo}}]\label{prop:7.1}
Assume that $\mathfrak{g}$ is of type $\text{A}_{n}$, $\text{B}_{n}$, or $\text{C}_{n}$. For $i \in I$, it holds that
\begin{equation}\label{quotient}
W^{J_{i}} = \{w_{n}(l_{n}) w_{n-1}(l_{n-1}) \cdots w_{i}(l_{i}) \mid l_{i}, \dots , l_{n-1}, l_{n}\ \text{satisfy condition}\ (\#)\},
\end{equation}
where condition $(\#)$ is given by (A) (resp., (BC1), (BC2), and (BC3)) below if $\mathfrak{g}$ is of type $\text{A}_{n}$ (resp., of type $\text{B}_{n}$ or $\text{C}_{n}$).
\begin{enumerate}
\item[(A)] $0 \leq l_{n} \leq l_{n-1} \leq \cdots \leq l_{i} \leq i$;
\item[(BC1)] $0 \leq l_{j} \leq j+i-1$,
\item[(BC2)] $l_{j+1} \leq l_{j} + 1$, and
\item[(BC3)] if $l_{j} \leq j-1$, then $l_{j+1} \leq l_{j}$.
\end{enumerate}
Moreover, for each element $w_{n}(l_{n}) w_{n-1}(l_{n-1}) \cdots w_{i}(l_{i})$ of the right-hand side of (\ref{quotient}), it holds that
\[
\ell (w_{n}(l_{n}) w_{n-1}(l_{n-1}) \cdots w_{i}(l_{i})) = \ell (w_{n}(l_{n})) + \ell (w_{n-1}(l_{n-1})) + \cdots + \ell (w_{i}(l_{i})).
\]
\end{prop}
\begin{prop}[{\cite[Theorem 4]{Stumbo}}]\label{prop:7.2}
Assume that $\mathfrak{g}$ is of type $\text{D}_{n}$. For $i \in I \backslash \{1,2\}$, it holds that
\begin{equation}\label{quotientD}
\begin{split}
W^{J_{i}} = \{w_{n}(l_{n})w_{n-1}&(l_{n-1}) \cdots w_{i}(l_{i}) \mid \\
&l_{i}, \dots, l_{n-1}, l_{n}\ \text{satisfy conditions}\ \text{(D1)--(D4)}\},
\end{split}
\end{equation}
where
\begin{enumerate}
\item[(D1)] $0 \leq l_{j} \leq j+i-2$,
\item[(D2)] $l_{j+1} \leq l_{j}+1$,
\item[(D3)] if $l_{j} \leq j-2$, then $l_{j+1} \leq l_{j}$, and
\item[(D4)] if $l_{j+1} = l_{j}+1 = j$, then $w_{j}(l_{j})$ and $w_{j+1}(l_{j+1})$ must be chosen in such a way that the one has $s_{1}$ as the left-most simple reflection, and the other has $s_{2}$.
\end{enumerate}
Moreover, for each element $w_{n}(l_{n}) w_{n-1}(l_{n-1}) \cdots w_{i}(l_{i})$ of the right-hand side of (\ref{quotientD}), it holds that
\[
\ell (w_{n}(l_{n}) w_{n-1}(l_{n-1}) \cdots w_{i}(l_{i})) = \ell (w_{n}(l_{n})) + \ell (w_{n-1}(l_{n-1})) + \cdots + \ell (w_{i}(l_{i})).
\]
For $i = 1$, it holds that
\begin{equation}\label{eq:Di1}
\begin{split}
W^{J_{1}} = \{w_{\frac{3+(-1)^{h}}{2}}(l_{h}) &w_{\frac{3+(-1)^{h-1}}{2}} (l_{h-1}) \cdots w_{2}(l_{4})w_{1}(l_{3})w_{2}(l_{2})w_{1}(l_{1}) \mid \\
&0 \leq h \leq n-1,\,  1 \leq l_{h} < l_{h-1} < \cdots < l_{1} \leq n-1\}.
\end{split}
\end{equation}
For $i = 2$, $W^{J_{2}}$ is given by the same formula as (\ref{eq:Di1}) with $w_{1}$ and $w_{2}$ interchanged.
Moreover, the ``length additivity'' holds also for the cases that $i = 1$ and $i = 2$.
\end{prop}

\begin{prop}\label{prop:7.4}
For $i \in K$, the set $\mathcal{SM}_{i} = \{w \in \mathcal{SM} \mid \Lambda_{w} = \Lambda_{i}\}$ (see Corollary \ref{rema:6.4}) is contained in $W^{J_{i}}$. If $\mathfrak{g}$ is of type $\text{A}_{n}$, $\text{B}_{n}$, or $\text{C}_{n}$, then it holds that
\begin{equation}\label{smi}
\mathcal{SM}_{i} = \{w_{n}(l_{n}) \cdots w_{i+1}(l_{i+1}) w_{i}(l_{i}) \mid l_{i}, \dots l_{n-1}, l_{n}\ \text{satisfy condition}\ (\star)\},
\end{equation}
where condition $(\star)$ is given by (SA) (resp., (SB), (SC)) below if $\mathfrak{g}$ is of type $\text{A}_{n}$ (resp., of type $\text{B}_{n}$, of type $\text{C}_{n}$).
\begin{enumerate}
\item[(SA)] Condition (A) in Proposition \ref{prop:7.1}, and $l_{i} = i$, $l_{n} \neq 0$;
\item[(SB)] Conditions (BC1)--(BC3) (with $i = 1$) and $l_{n} \neq 0$;
\item[(SC)] If $2 \leq i \leq n-1$, then $i \leq l_{i} \leq 2i-2$ and $1 \leq l_{n} \leq \cdots \leq l_{i+1} \leq 2i-l_{i}-1$. If $i = n$, then $n \leq l_{n} \leq 2n-1$.
\end{enumerate}
\noindent
Also, if $\mathfrak{g}$ is of type $\text{D}_{n}$, then it holds that
\begin{equation}\label{eq:DS1}
\begin{split}
\mathcal{SM}_{1} &= \{w_{\frac{3+(-1)^{h}}{2}}(l_{h}) w_{\frac{3+(-1)^{h-1}}{2}} (l_{h-1}) \cdots w_{2}(l_{4})w_{1}(l_{3})w_{2}(l_{2})w_{1}(l_{1}) \mid \\
&\ \ \ \ \ \ \ \ \ \ \ \ \ 2 \leq h \leq n-1,\, 1 \leq l_{h} < l_{h-1} < \cdots < l_{2} < l_{1} = n-1\}. \\
\end{split}
\end{equation}
For $i=2$, $\mathcal{SM}_{2}$ is given by the same formula as (\ref{eq:DS1}) with $w_{1}$ and $w_{2}$ interchanged. Moreover, $\mathcal{SM}_{i}$, $3 \leq i \leq n-1$, and $\mathcal{SM}_{n}$ are given as follows:
\begin{equation}\label{Dsmi}
\mathcal{SM}_{i} = \{w_{n}(l_{n}) \cdots w_{i}(l_{i}) \mid i \leq l_{i} \leq 2i-3, 1 \leq l_{n} \leq \cdots \leq l_{i+1} \leq 2i-l_{i}-2\},
\end{equation}
\begin{equation}\label{Dsmn}
\mathcal{SM}_{n} = \{w_{n}(l_{n}) \mid n \leq l_{n} \leq 2n-2\}.
\end{equation}
\end{prop}

\begin{proof}
We give a proof only (\ref{Dsmi}); the proofs for (\ref{smi}), (\ref{eq:DS1}), and (\ref{Dsmn}) are similar and simpler. In order to show the inclusion $\subset$, let $w \in \mathcal{SM}_{i}$. By Lemma \ref{lemm:6.2}, in any reduced expression of $w$, the right-most generator is $s_{i}$. Hence, we have $w \in W^{J_{i}}$ by \cite[Lemma 2.4.3]{BB}. By Proposition \ref{prop:7.2}, we can write $w$ as 
\begin{equation}\label{eq:7.2}
w = w_{n}(l_{n}) \cdots w_{i}(l_{i})
\end{equation}
for some $l_{i}, \dots , l_{n-1}, l_{n}$ satisfying conditions (D1)--(D4). If $l_{j} = 0$ for some $i \leq j \leq n$, then $l_{n} = l_{n-1} = \cdots = l_{j+1} = 0$, which implies that $s_{n}$ does not appear in (\ref{eq:7.2}). However, this contradicts Proposition \ref{prop:4.2}. Thus we obtain $l_{j} \geq 1$ for all $i \leq j \leq n$. Let $w = s_{j_{r}} \cdots s_{j_{1}}$ be a reduced expression of $w$ obtained by the product of reduced expressions of each $w_{j}(l_{j})$ in (\ref{eq:7.2}). Suppose, for a contradiction, that $l_{i} = 2i-2$. Since $l_{i+1} \geq 1$ and $j_{2i-1} = i+1 \neq i$, this contradicts (\ref{align:6.12}) because $u_{2i-1} = 2 \notin 2\mathbb{N}-1$. Hence we have $l_{i} \leq 2i-3$. Next, let us show that $i \leq l_{i}$. If $l_{i} \leq i-1$, then we have $l_{i}=i-1$ and $l_{i+1}=i$ because both $s_{1}$ and $s_{2}$ appear in (\ref{eq:7.2}) by Proposition \ref{prop:4.2}. Since $j_{2i-1} = 1\neq i$ (or $j_{2i-1} = 2 \neq i$), this contradicts (\ref{align:6.12}) because $u_{2i-1} = 2 \notin 2\mathbb{N}-1$. Therefore, we have $i \leq l_{i} \leq 2i-3$.
Suppose, for a contradiction, that $l_{i+1} \geq 2i-l_{i}-1$. Since $j_{2i-1} = l_{i}-i-3$, it follows that $u_{2i-1} = 4 \notin 2\mathbb{N}-1$ (resp., $u_{2i-1} = 3 \notin 2\mathbb{N}$) if $i \leq l_{i} \leq 2i-4$ (resp., $l_{i} = 2i-3$). This contradicts (\ref{align:6.12}) (resp., (\ref{align:6.11})). Hence we have $l_{i+1} \leq 2i-l_{i}-2$.
Recall from (D2) that $l_{j+1} \leq l_{j}+1$ for all $i+1 \leq j \leq n-1$. Suppose, for a contradiction, that $l_{j+1} = l_{j} + 1$ for some $i+1 \leq j \leq n-1$. If we set $m \coloneqq \text{min} \{i+1 \leq j \leq n-1 \mid l_{j+1} = l_{j} + 1\}$, then $l_{m} \leq l_{m-1} \leq \cdots \leq l_{i+1}$. By direct computation, we obtain
\[
u_{M} =
\begin{cases}
2l_{m}\ &\text{if}\ \ l_{m} \leq m-i+1, \\
2(m-i)+1\ &\text{if}\ \ l_{m} = m-i+1, \\
2(m-i+1)\ &\text{if}\ \ l_{m} \geq m-i+1,
\end{cases}
\]
where $M \coloneqq l_{m+1} + l_{m} + \cdots + l_{i}$; remark that $l_{m} = m-i+1$ if and only if $j_{M} = i$. This contradicts (\ref{align:6.11}) and (\ref{align:6.12}). Therefore we obtain $1 \leq l_{n} \leq l_{n-1} \leq \cdots \leq l_{i+1} \leq 2i-l_{i}-2$, as desired. Thus we have shown the inclusion $\subset$.

Next, let us show the reverse inclusion $\supset$. Let $3 \leq i \leq n-1$, and let $w = w_{n}(l_{n}) \cdots w_{i}(l_{i})$ with $i \leq l_{i} \leq 2i-3$ and $1 \leq l_{n} \leq \cdots \leq l_{i+1} \leq 2i-l_{i}-2$. Set $k_{i} \coloneqq l_{i}-i+2$; note that $2 \leq k_{i} \leq i-1$. Take $\varepsilon_{i} \in \mathfrak{h}^{\ast}$, $i \in I$, such that $\alpha_{1} = \varepsilon_{2} + \varepsilon_{1}$, $\alpha_{j} = \varepsilon_{j} - \varepsilon_{j-1}$ for $2 \leq j \leq n$, $\Lambda_{1} = \frac{1}{2}(\varepsilon_{1} + \varepsilon_{2} + \cdots + \varepsilon_{n})$, $\Lambda_{2} = \frac{1}{2}(-\varepsilon_{1} + \varepsilon_{2} + \cdots + \varepsilon_{n})$ and $\Lambda_{j} = \varepsilon_{j} + \varepsilon_{j+1} + \cdots + \varepsilon_{n}$ for $3 \leq j \leq n$. Then, we compute
\begin{align*}
w_{i}(l_{i}) \Lambda_{i} &= s_{k_{i}} \cdots s_{3}s_{1}s_{2}s_{3} \cdots s_{i-1}\underbrace{s_{i} (\varepsilon_{n} + \varepsilon_{n-1} + \cdots + \varepsilon_{i+1} + \varepsilon_{i})}_{(\varepsilon_{n} + \cdots + \varepsilon_{i+1} + \varepsilon_{i}, \varepsilon_{i} - \varepsilon_{i-1}) = 1} \\
&= s_{k_{i}} \cdots s_{3}s_{1}s_{2}s_{3} \cdots \underbrace{s_{i-1} (\varepsilon_{n} + \varepsilon_{n-1} + \cdots + \varepsilon_{i+1} + \varepsilon_{i-1})}_{(\varepsilon_{n} + \cdots + \varepsilon_{i+1} + \varepsilon_{i-1}, \varepsilon_{i-1} - \varepsilon_{i-2}) = 1} \\
&= \cdots \cdots \cdots \cdots \\
&= s_{k_{i}} \cdots s_{3}\underbrace{s_{1} (\varepsilon_{n} + \varepsilon_{n-1} + \cdots + \varepsilon_{i+1} + \varepsilon_{1})}_{(\varepsilon_{n} + \cdots + \varepsilon_{i+1} + \varepsilon_{1}, \varepsilon_{2} + \varepsilon_{1}) = 1} \\
&= s_{k_{i}} \cdots \underbrace{s_{3}(\varepsilon_{n} + \varepsilon_{n-1} + \cdots + \varepsilon_{i+1} - \varepsilon_{2})}_{(\varepsilon_{n} + \cdots + \varepsilon_{i+1} - \varepsilon_{2}, \varepsilon_{3} - \varepsilon_{2}) = 1} \\
&= \cdots \cdots \cdots \cdots \\
&= \underbrace{s_{k_{i}}(\varepsilon_{n} + \varepsilon_{n-1} + \cdots + \varepsilon_{i+1} - \varepsilon_{k_{i}-1})}_{(\varepsilon_{n} + \cdots + \varepsilon_{i+1} - \varepsilon_{k_{i}-1}, \varepsilon_{k_{i}} - \varepsilon_{k_{i}-1}) = 1} \\
&= \varepsilon_{n} + \varepsilon_{n-1} + \cdots + \varepsilon_{i+1} - \varepsilon_{k_{i}}.
\end{align*}
Since $1 \leq l_{n} \leq \cdots \leq l_{i+1} \leq 2i-l_{i}-2 \leq i-2$, we can write $w_{j}(l_{j})$ as $w_{j}(l_{j}) = s_{p_{j}} s_{p_{j}+1} \cdots s_{j-1}s_{j}$, where $p_{j} \coloneqq j-l_{j}+1$ for $i+1 \leq j \leq n$; remark that $p_{j} \leq j$ and $k_{i} + 1 < p_{i+1} < p_{i+2} < \cdots < p_{n} \leq n$. We compute
\begin{align*}
w_{i+1}(l_{i+1}) (\varepsilon_{n} + \cdots + \varepsilon_{i+1}-\varepsilon_{k_{i}}) &= s_{p_{i+1}}s_{p_{i+1}+1} \cdots s_{i}\underbrace{s_{i+1} (\varepsilon_{n} + \cdots + \varepsilon_{i+1}-\varepsilon_{k_{i}})}_{(\varepsilon_{n} + \cdots + \varepsilon_{i+1} - \varepsilon_{k_{i}}, \varepsilon_{i+1}-\varepsilon_{i}) = 1} \\
&= s_{p_{i+1}}s_{p_{i+1}+1} \cdots \underbrace{s_{i}(\varepsilon_{n} + \cdots + \varepsilon_{i+2} + \varepsilon_{i} - \varepsilon_{k_{i}})}_{(\varepsilon_{n} + \cdots + \varepsilon_{i+2} + \varepsilon_{i} - \varepsilon_{k_{i}}, \varepsilon_{i}-\varepsilon_{i-1}) = 1} \\
&= \cdots \cdots \cdots \cdots \\
&= \underbrace{s_{p_{i+1}} (\varepsilon_{n} + \cdots + \varepsilon_{i+2} + \varepsilon_{p_{i+1}} - \varepsilon_{k_{i}})}_{\mathclap{(\varepsilon_{n} + \cdots + \varepsilon_{i+2} + \varepsilon_{p_{i+1}} - \varepsilon_{k_{i}}, \varepsilon_{p_{i+1}}-\varepsilon_{p_{i+1}-1}) = 1}} \\
&= \varepsilon_{n} + \cdots + \varepsilon_{i+2} + \varepsilon_{p_{i+1}-1} - \varepsilon_{k_{i}},
\end{align*}
which implies that $w_{i+1}(l_{j+1})w_{i}(l_{i})$ is $\Lambda_{i}$-minuscule. Similarly,  we see that for $i+1 \leq j \leq n-2$,
\begin{equation*}
\begin{split}
w_{j+1}(l_{j+1}) (&\varepsilon_{n} + \cdots + \varepsilon_{j+1} + \varepsilon_{p_{j}-1} + \varepsilon_{p_{j-1}-1} + \cdots + \varepsilon_{p_{i+1}-1} - \varepsilon_{k_{i}}) \\
=\ &\varepsilon_{n} + \cdots + \varepsilon_{j+2} + \varepsilon_{p_{j+1}-1} + \varepsilon_{p_{j}-1} + \cdots + \varepsilon_{p_{i+1}-1} - \varepsilon_{k_{i}},
\end{split}
\end{equation*}
and hence $w_{j+1}(l_{j+1}) \cdots w_{i+1}(l_{i+1})w_{i}(l_{i})$ is $\Lambda_{i}$-minuscule. Then,
\begin{align*}
&w_{n}(l_{n}) (\varepsilon_{n} + \varepsilon_{p_{n-1}-1} + \cdots + \varepsilon_{p_{i+1}-1} - \varepsilon_{k_{i}}) \\
=\ &s_{p_{n}}s_{p_{n}+1} \cdots s_{n-1}\underbrace{s_{n}(\varepsilon_{n} + \varepsilon_{p_{n-1}-1} + \cdots + \varepsilon_{p_{i+1}-1} - \varepsilon_{k_{i}})}_{(\varepsilon_{n} + \varepsilon_{p_{n-1}-1} + \cdots + \varepsilon_{p_{i+1}-1} - \varepsilon_{k_{i}}, \varepsilon_{n}-\varepsilon_{n-1}) = 1} \\
=\ &\cdots \cdots \cdots \cdots \\
=\ &\underbrace{s_{p_{n}}(\varepsilon_{p_{n}} + \varepsilon_{p_{n-1}-1} + \cdots + \varepsilon_{p_{i+1}-1} - \varepsilon_{k_{i}})}_{\mathclap{(\varepsilon_{p_{n}} + \varepsilon_{p_{n-1}-1} + \cdots + \varepsilon_{p_{i+1}-1} - \varepsilon_{k_{i}}, \varepsilon_{p_{n}} - \varepsilon_{p_{n}-1}) = 1}},
\end{align*}
which implies $w = w_{n}(l_{n}) \cdots w_{i+1}(l_{i+1}) w_{i}(l_{i})$ is $\Lambda_{i}$-minuscule.

Finally, let us show that $w = w_{n}(l_{n}) \cdots w_{i+1}(l_{i+1}) w_{i}(l_{i})$ is a strong minuscule element. In the expression $w = w_{n}(l_{n}) \cdots w_{i+1}(l_{i+1}) w_{i}(l_{i})$, we move the right-most $s_{j}$ in each $w_{j}(l_{j})$ to the right position, by using the commutation relation $s_{p}s_{q} = s_{q}s_{p}$ for $3 \leq p, q \leq n$ with $|p - q| \geq 2$, as follows:
\begin{align*}
&\overbrace{s_{n-l_{n}+1}\cdots s_{n-1}s_{n}}^{w_{n}(l_{n}) =}\overbrace{\underbrace{s_{(n-1)-l_{n-1}+1} \cdots s_{n-2}}_{\text{these commute with}\ s_{n}}s_{n-1}}^{w_{n-1}(l_{n-1}) =} w_{n-2}(l_{n-2}) \cdots w_{i}(l_{i}) \\
=&\ (w_{n}(l_{n}) s_{n})(w_{n-1}(l_{n-1}) s_{n-1}) s_{n}s_{n-1} \overbrace{\underbrace{s_{(n-2)-l_{n-2}+1} \cdots s_{n-3}}_{\text{these\ commute\ with\ $s_{n}s_{n-1}$}}s_{n-2}}^{w_{n-2}(l_{n-2}) =} \cdots w_{i}(l_{i}) \\
=&\ (w_{n}(l_{n}) s_{n}) (w_{n-1}(l_{n-1}) s_{n-1}) (w_{n-2}(l_{n-2}) s_{n-2}) s_{n}s_{n-1}s_{n-2} w_{n-3}(l_{n-3}) \cdots w_{i}(l_{i}) \\
=&\ \cdots \cdots \cdots \cdots \\
=&\ \underbrace{(w_{n}(l_{n}) s_{n}) (w_{n-1}(l_{n-1}) s_{n-1}) \cdots (w_{i+1}(l_{i+1}) s_{i+1})}_{\eqqcolon u'} s_{n} \cdots s_{i+1} w_{i}(l_{i}) \\
=&\ u' s_{n} \cdots s_{i+1} \underbrace{s_{k_{i}} \cdots s_{3}}_{\mathclap{\text{these\ commute\ with}\ s_{n} \cdots s_{i+1}}} s_{1}s_{2}s_{3} \cdots s_{i} \\
=&\ \underbrace{u' s_{k_{i}} \cdots s_{3}}_{\eqqcolon u} \underbrace{s_{n} \cdots s_{i+1} s_{1}s_{2}s_{3} \cdots s_{i}}_{=v_{i}};
\end{align*}
remark that if $i = 3$, then $u = e$.
Therefore it follows from Lemma \ref{lemm:6.2} that $w = w_{n}(l_{n}) \cdots w_{i+1}(l_{i+1}) w_{i}(l_{i})$ is a strong minuscule element. This completes the proof of Proposition \ref{prop:7.4}.
\end{proof}

\begin{prop}\label{prop:enumerate}
It hold that \\
(i) If $\mathfrak{g}$ is of type $\text{A}_{n}$, then $\# \mathcal{SM}_{i} = \binom{n-1}{i-1}$ for $1 \leq i \leq n$; \\
(ii) If $\mathfrak{g}$ is of type $\text{B}_{n}$, then $\# \mathcal{SM}_{1} = 2^{n-1}$; \\
(iii) If $\mathfrak{g}$ is of type $\text{C}_{n}$, then $\# \mathcal{SM}_{i} = \binom{n-1}{i-2}$ for $2 \leq i \leq n-1$, and $\# \mathcal{SM}_{n} = n$; \\
(iv) If $\mathfrak{g}$ is of type $\text{D}_{n}$, then $\# \mathcal{SM}_{1} = \# \mathcal{SM}_{2} = 2^{n-2}-1$, $\# \mathcal{SM}_{i} = \binom{n-2}{i-3}$ for $3 \leq i \leq n-1$, and $\# \mathcal{SM}_{n} = n-1$.
\end{prop}

\begin{proof}
We give proofs only for the cases of type $\text{B}_{n}$ and type $\text{C}_{n}$; the proofs for the other cases are similar or simpler. In this proof, we denote by $W(\text{B}_{n})$ the Weyl group of type $\text{B}_{n}$, and set $J_{i}^{(n)} \coloneqq \{s_{1}, \dots, s_{n}\} \backslash \{s_{i}\}$. In the case of type $\text{B}_{n}$, we see from Proposition \ref{prop:7.4} that $\mathcal{SM}_{1} = \{w_{n}(l_{n}) \cdots w_{1}(l_{1}) \mid l_{n}, \dots , l_{1}\ \text{satisfy}\ l_{n} \neq 0\ \text{and}\ \text{(BC1)--(BC3)}\}$. It is easy to see by Proposition \ref{prop:7.1} that
\[
\# \{w_{n}(l_{n}) \cdots w_{1}(l_{1}) \mid l_{n}, \dots , l_{1}\ \text{satisfy}\ l_{n} = 0\ \text{and}\ \text{(BC1)--(BC3)}\} = \# W(\text{B}_{n-1})^{J_{1}^{(n-1)}}.
\]
Therefore, we obtain
\begin{align*}
\# \mathcal{SM}_{1} &= \# W(\text{B}_{n})^{J_{1}^{(n)}} - \# W(\text{B}_{n-1})^{J_{1}^{(n-1)}} = \frac{\# W(\text{B}_{n})}{\# W(\text{B}_{n})_{J_{1}^{(n)}}} - \frac{\# W(\text{B}_{n-1})}{\# W(\text{B}_{n-1})_{J_{1}^{(n-1)}}} \\
&= \frac{n! \times 2^{n}}{n!} - \frac{(n-1)! \times 2^{n-1}}{(n-1)!} = 2^{n} - 2^{n-1} = 2^{n-1},
\end{align*}
as desired.

In the case of type $\text{C}_{n}$ with $2 \leq i \leq n-1$, we see from Proposition \ref{prop:7.4} that $\mathcal{SM}_{i} = \{w_{n}(l_{n}) \cdots w_{i}(l_{i}) \mid i \leq l_{i} \leq 2i-2, 1 \leq l_{n} \leq \cdots \leq l_{i+1} \leq 2i-l_{i}-1\}$. Hence we have
\begin{align*}
\# \mathcal{SM}_{i} &= \sum_{l_{i}=i}^{2i-2} \binom{n+i-2-l_{i}}{n-i} = \binom{n-2}{n-i} + \binom{n-3}{n-i} + \cdots + \binom{n-i}{n-i} \\
&= \binom{n}{i-1}-\binom{n-1}{i-1} = \binom{n-1}{i-2};
\end{align*}
remark that $\binom{n}{r} = \sum_{k=r-1}^{n-1} \binom{k}{r-1}$ and $\binom{n}{r} = \binom{n-1}{r} + \binom{n-1}{r-1}$.
\end{proof}

\section{Application to Demazure modules.}\label{sec:8}
\subsection{Bruhat order.}
We denote by $\leq$ the Bruhat order on $W$ (see, e.g., \cite[Chapter 2]{BB}). For $u, w \in W$, we set $[u, w] \coloneqq \{v \in W \mid u \leq v \leq w\}$. Denote by $w_{0}$ the longest element in $W$; note that $w \leq w_{0}$ for all $w \in W$.
Let $J \subset S$; recall from Section \ref{sec:7} that $W^{J} (\subset W)$ denotes the set of minimal-length coset representatives of cosets in $W / W_{J}$. Let $w_{0}^{J} \in W^{J}$ be such that $w_{0}^{J} \in w_{0}W_{J}$. Then, $w \leq w_{0}^{J}$ for all $w \in W^{J}$ (see \cite[Section 2.5]{BB}).
For $u, w \in W^{J}$, we set $[u, w]^{J} \coloneqq [u, w] \cap W^{J}$.

\begin{prop}\label{prop:9}
Let $\mathcal{SM}_{i}$ be as in Remark \ref{rema:6.4} (see also Proposition \ref{prop:7.4}). It hold that \\
(i') If $\mathfrak{g}$ is of type $\text{A}_{n}$, then $\mathcal{SM}_{i} = [v_{i}, w_{0}^{J_{i}}]^{J_{i}}$ for $1 \leq i \leq n$; \\
(ii') If $\mathfrak{g}$ is of type $\text{B}_{n}$, then $\mathcal{SM}_{1} = [v_{1}, w_{0}^{J_{1}}]^{J_{1}}$; \\
(iii') If $\mathfrak{g}$ is of type $\text{C}_{n}$, then $\mathcal{SM}_{n} = [v_{n}, w_{0}^{J_{n}}]^{J_{n}} \backslash \{w_{0}^{J_{n}}\}$; \\
(iv') If $\mathfrak{g}$ is of type $\text{D}_{n}$, then $\mathcal{SM}_{1} = [v_{1}, w_{0}^{J_{1}}]^{J_{1}}$, $\mathcal{SM}_{2} = [v_{2}, w_{0}^{J_{2}}]^{J_{2}}$, and $\mathcal{SM}_{n} = [v_{n}, w_{0}^{J_{n}}]^{J_{n}} \backslash \{w_{0}^{J_{n}}\}$.
\end{prop}

\begin{proof}
We give a proof only for the case of type $\text{A}_{n}$; the proofs for the other cases are similar or simpler. Let $w \in \mathcal{SM}_{i}$. By Proposition \ref{prop:7.4}, we have $w \in W^{J_{i}}$. Hence, we have $w \leq w_{0}^{J_{i}}$. By Lemma \ref{lemm:6.2}, there exists $u \in W$ such that $w = uv_{i}$ with $\ell (w) = \ell (u) + \ell (v_{i})$. Hence, by the subword property of the Bruhat order (see, e.g., \cite[Theorem 2.2.2]{BB}), we have $v_{i} \leq w$. Therefore, we conclude that $w \in [v_{i}, w_{0}^{J_{i}}]^{J_{i}}$.

Conversely, let $w \in [v_{i}, w_{0}^{J_{i}}]^{J_{i}} = [v_{i}, w_{0}^{J_{i}}] \cap W^{J_{i}}$. By Proposition \ref{prop:7.1}, there exist $0 \leq p_{n} \leq \cdots \leq p_{i} \leq i$ such that $w = w_{n}(p_{n}) \cdots w_{i}(p_{i})$. Since $v_{i} \leq w$ by assumption, it follows from the subword property that both $s_{1}$ and $s_{n}$ appear in any reduced expression for $w$. Observe that for $i < j \leq n$, the element $w_{j}(p_{j})$ does not have a reduced expression in which $s_{1}$ appears, and that the element $w_{i}(p_{i})$ has a reduced expression in which $s_{1}$ appears if and only if $p_{i} = i$. Thus we conclude that $p_{i} = i$. Also, observe that for $i \leq j < n$, the element $w_{j}(p_{j})$ does not have a reduced expression in which $s_{n}$ appears, and that the element $w_{n}(p_{n})$ has a reduced expression in which $s_{n}$ appears if and only if $p_{n} \geq 1$. Thus we conclude that $p_{n} \geq 1$. Therefore, by Proposition \ref{prop:7.4}, we have $w \in \mathcal{SM}_{i}$, as desired.
\end{proof}

\begin{rema}
In general, $[v_{i}, w_{0}^{J_{i}}]^{J_{i}} \subsetneq [v_{i}, w_{0}^{J_{i}}]$. Indeed, in the Weyl group of type $\text{A}_{4}$,
we see that $s_{2}v_{3} = s_{2}s_{1}s_{2}s_{4}s_{3} \in [v_{3}, w_{0}^{J_{3}}] \backslash [v_{3}, w_{0}^{J_{3}}]^{J_{3}}$;
note that this element is not a minuscule element, and hence Lemma \ref{lemm:6.2} is not valid for this element.
\end{rema}

\subsection{Demazure modules.}\label{ss:1}
For $\Lambda \in P^{+}$, let $L(\Lambda)$ denote the finite-dimensional irreducible $\mathfrak{g}$-module of highest weight $\Lambda$, with $L(\Lambda) = \bigoplus_{\mu \in P} L(\Lambda)_{\mu}$ the weight space decomposition; recall that dim $L(\Lambda)_{\tau(\Lambda)} = 1$ for all $\tau \in W$. Denote by $\mathfrak{n}_{+}$ the subalgebra of $\mathfrak{g}$ generated by the root spaces corresponding to $\Phi_{+}$.
For $\tau \in W$, we denote by $E_{\tau} (\Lambda)$ the $\mathfrak{n}_{+}$-submodule of $L(\Lambda)$ generated by $L(\Lambda)_{\tau(\Lambda)}$,
which we call the {\it Demazure module} of lowest weight $\tau(\Lambda)$.

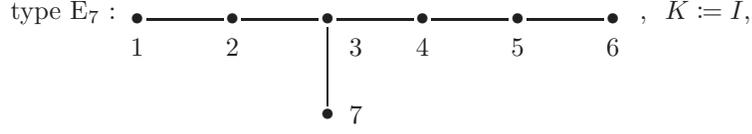
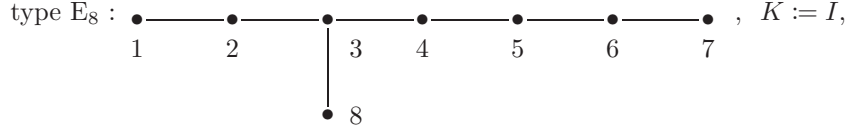
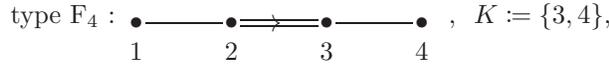
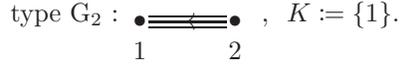
\begin{rema} \label{lema:11111}
For $i \in I$, we assume that $\Lambda$ = $\Lambda_{i}$ is a minuscule weight in the sense that $\langle \Lambda_{i}, \beta^{\vee} \rangle \in \{0, \pm 1\}$ for all $\beta \in \Phi$, and $J_{\Lambda_{i}}$ = $\{s_{j} \in S\ |\ \langle \Lambda_{i}, \alpha_{j}^{\vee} \rangle = 0\}$ = $S \backslash \{s_{i}\}$ = $J_{i}$. In this case, the dimension of the Demazure module $E_{\tau} (\Lambda)$ for $\tau \in W^{J_{i}}$ is equal to $[e, \tau]^{J_{i}}$ (this fact follows from, for example, the theory of Lakshmibai-Seshadri paths; see \cite[Theorem 5.2]{Littelmann}).
\end{rema}

\subsection{Dimension formula for Demazure modules.}
Let and fix $i \in I$. For $\tau \in W^{J_{i}}$, we set $\bar{\tau} \coloneqq w_{0} \tau w_{J_{i}, 0}$, where $w_{J_{i}, 0} \in W_{J_{i}}$
is the longest element of $W_{J_{i}}$. Then we see by \cite[Proposition 2.5.4]{BB} that $\bar{\tau} \in W^{J_{i}}$,
and that the map $\overline{\,\cdot\,} : W^{J_{i}} \to W^{J_{i}}$, $\tau \mapsto \bar{\tau}$, is an order-reversing involution on $W^{J_{i}}$.
\begin{thm}\label{thm:5}
It hold that \\
(1) If $\mathfrak{g}$ is of type $\text{A}_{n}$, then $\text{dim}\, E_{\overline{v_{i}}}(\Lambda_{i}) = \binom{n-1}{i-1}$ for each $i \in I$; \\
(2) If $\mathfrak{g}$ is of type $\text{B}_{n}$, then $\text{dim}\, E_{\overline{v_{1}}}(\Lambda_{1}) = 2^{n-1}$; \\
(3) If $\mathfrak{g}$ is of type $\text{C}_{n}$, then $\text{dim}\, E_{\overline{v_{n}}}(\Lambda_{n}) = n+1$; \\
(4) If $\mathfrak{g}$ is of type $\text{D}_{n}$, then $\text{dim}\, E_{\overline{v_{1}}}(\Lambda_{1}) = \text{dim}\, E_{\overline{v_{2}}}(\Lambda_{2}) = 2^{n-2}-1$, and $\text{dim}\, E_{\overline{v_{n}}}(\Lambda_{n}) = n$.
\end{thm}

\begin{proof}
We see that
\begin{align*}
\overline{[v_{i}, w_{0}^{J_{i}}]^{J_{i}}} = [\overline{w_{0}^{J_{i}}}, \overline{v_{i}}]^{J_{i}} = [w_{0} w_{0}^{J_{i}} w_{J_{i}, 0}, \overline{v_{i}}]^{J_{i}} = [w_{0}^{2}, \overline{v_{i}}]^{J_{i}} = [e, \overline{v_{i}}]^{J_{i}}.
\end{align*}
Hence, $\# [v_{i}, w_{0}^{J_{i}}]^{J_{i}} = \# [e, \overline{v_{i}}]^{J_{i}}$. Because we have $\# [v_{i}, w_{0}^{J_{i}}]^{J_{i}} = \# \mathcal{SM}_{i}$ or $\# [v_{i}, w_{0}^{J_{i}}]^{J_{i}} = \# \mathcal{SM}_{i} + 1$ by Propositions \ref{prop:enumerate} and \ref{prop:9}, we conclude by using Remark \ref{lema:11111} that $\text{dim}\, E_{\overline{v_{i}}} (\Lambda_{i}) = \# [e, \overline{v_{i}}]^{J_{i}} = \# [v_{i}, w_{0}^{J_{i}}]^{J_{i}} = \# \mathcal{SM}_{i}$ or $\# \mathcal{SM}_{i} + 1$, as desired.
\end{proof}

\appendix
\section{Appendix.}

In this appendix, we assume that $\mathfrak{g}$ is the exceptional finite-dimensional simple Lie algebra of type $\text{E}_{6}$, $\text{E}_{7}$, $\text{E}_{8}$, $\text{F}_{4}$, or $\text{G}_{2}$. The Dynkin diagram for $\mathfrak{g}$ and $K \subset I$ are given as follows:

\vspace{5mm}
type $\text{E}_{6}$ :
\xygraph{
    \bullet ([]!{+(0,-.3)} {1}) - [r]
    \bullet ([]!{+(0,-.3)} {2}) - [r]
    \bullet ([]!{+(.3,-.3)} {3}) (
        - [d] \bullet ([]!{+(.3,0)} {6}),
        - [r] \bullet ([]!{+(0,-.3)} {4})
        - [r] \bullet ([]!{+(0,-.3)} {5}))}\ \  ,\ \ 
$K \coloneqq I$,

\vspace{5mm}
type $\text{E}_{7}$ :
\xygraph{
    \bullet ([]!{+(0,-.3)} {1}) - [r]
    \bullet ([]!{+(0,-.3)} {2}) - [r]
    \bullet ([]!{+(.3,-.3)} {3}) (
        - [d] \bullet ([]!{+(.3,0)} {7}),
        - [r] \bullet ([]!{+(0,-.3)} {4})
        - [r] \bullet ([]!{+(0,-.3)} {5})
        - [r] \bullet ([]!{+(0,-.3)} {6}))}\ \ ,\ \ 
$K \coloneqq I$,

\vspace{5mm}
type $\text{E}_{8}$ :
\xygraph{
    \bullet ([]!{+(0,-.3)} {1}) - [r]
    \bullet ([]!{+(0,-.3)} {2}) - [r]
    \bullet ([]!{+(.3,-.3)} {3}) (
        - [d] \bullet ([]!{+(.3,0)} {8}),
        - [r] \bullet ([]!{+(0,-.3)} {4})
        - [r] \bullet ([]!{+(0,-.3)} {5})
        - [r] \bullet ([]!{+(0,-.3)} {6})
        - [r] \bullet ([]!{+(0,-.3)} {7}))}\ \ ,\ \ 
$K \coloneqq I$,

\vspace{5mm}
type $\text{F}_{4}$ :
\xygraph{!~:{@{=}|@{>}}
    \bullet ([]!{+(0,-.3)} {1}) - [r]
    \bullet ([]!{+(0,-.3)} {2}) : [r]
    \bullet ([]!{+(0,-.3)} {3}) - [r]
    \bullet ([]!{+(0,-.3)} {4})}\ \ ,\ \ 
$K \coloneqq \{3, 4\}$,

\vspace{7mm}
type $\text{G}_{2}$ :
\xygraph{!~:{@3{-}|@{<}}
    \bullet ([]!{+(0,-.3)} {1}) : [r]
    \bullet ([]!{+(0,-.3)} {2})}\ \ ,\ \ 
$K \coloneqq \{1\}$.

\vspace{5mm}
\begin{table}[h]
\centering
\begin{tabular}{c|c|c|c|c|c}
$v_{i}$ & $\text{E}_{6}$ & $\text{E}_{7}$ & $\text{E}_{8}$ & $\text{F}_{4}$ & $\text{G}_{2}$ \\ \hline \hline
$i=1$ & $s_{6}s_{5}s_{4}s_{3}s_{2}s_{1}$ & $s_{7}s_{6}s_{5}s_{4}s_{3}s_{2}s_{1}$ & $s_{8}s_{7}s_{6}s_{5}s_{4}s_{3}s_{2}s_{1}$ & \bf{--} & $s_{2}s_{1}$ \\ \hline
$i=2$ & $s_{6}s_{5}s_{4}s_{3}s_{1}s_{2}$ & $s_{7}s_{6}s_{5}s_{4}s_{3}s_{1}s_{2}$ & $s_{8}s_{7}s_{6}s_{5}s_{4}s_{3}s_{1}s_{2}$ & \bf{--} & \bf{--} \\ \hline
$i=3$ & $s_{6}s_{5}s_{4}s_{1}s_{2}s_{3}$ & $s_{7}s_{6}s_{5}s_{4}s_{1}s_{2}s_{3}$ & $s_{8}s_{7}s_{6}s_{5}s_{4}s_{1}s_{2}s_{3}$ & $s_{1}s_{2}s_{4}s_{3}$ & \bf{--} \\ \hline
$i=4$ & $s_{6}s_{5}s_{1}s_{2}s_{3}s_{4}$ & $s_{7}s_{6}s_{5}s_{1}s_{2}s_{3}s_{4}$ & $s_{8}s_{7}s_{6}s_{5}s_{1}s_{2}s_{3}s_{4}$ & $s_{1}s_{2}s_{3}s_{4}$ & \bf{--} \\ \hline
$i=5$ & $s_{6}s_{1}s_{2}s_{3}s_{4}s_{5}$ & $s_{7}s_{6}s_{1}s_{2}s_{3}s_{4}s_{5}$ & $s_{8}s_{7}s_{6}s_{1}s_{2}s_{3}s_{4}s_{5}$ & \bf{--} & \bf{--} \\ \hline
$i=6$ & $s_{1}s_{2}s_{5}s_{4}s_{3}s_{6}$ & $s_{7}s_{1}s_{2}s_{3}s_{4}s_{5}s_{6}$ & $s_{8}s_{7}s_{1}s_{2}s_{3}s_{4}s_{5}s_{6}$ & \bf{--} & \bf{--} \\ \hline
$i=7$ & \bf{--} & $s_{1}s_{2}s_{6}s_{5}s_{4}s_{3}s_{7}$ & $s_{8}s_{1}s_{2}s_{3}s_{4}s_{5}s_{6}s_{7}$ & \bf{--} & \bf{--} \\ \hline
$i=8$ & \bf{--} & \bf{--} & $s_{1}s_{2}s_{7}s_{6}s_{5}s_{4}s_{3}s_{8}$ & \bf{--} & \bf{--}
\end{tabular}
\caption{An element $v_{i} \in W$.}
\end{table}
Define $v_{i} \in W$ for $i \in K$ as Table 1 above. Then we deduce that the same statements as Lemma \ref{lemm:6.2} and Corollary \ref{rema:6.4} hold also in these exceptional cases. In particular, we have
\[
\mathcal{SM} = \bigsqcup_{i \in K} \mathcal{SM}_{i},
\]
where $\mathcal{SM}_{i} = \{w \in \mathcal{SM} \mid \Lambda_{w} = \Lambda_{i}\}$ for $i \in K$.
By using a computer, we see that $\# \mathcal{SM}_{i}$ is given as Table 2 below.
Furthermore, if $\mathfrak{g}$ is of type $\text{E}_{6}$, then $\mathcal{SM}_{1} = [v_{1}, w_{0}^{J_{1}}]^{J_{1}}$ and $\mathcal{SM}_{5} = [v_{5}, w_{0}^{J_{5}}]^{J_{5}}$. Hence we have $\text{dim}\, E_{\overline{v_{1}}} (\Lambda_{1}) = 16$ and $\text{dim}\, E_{\overline{v_{5}}} (\Lambda_{5}) = 16$. Also, if $\mathfrak{g}$ is of type $\text{E}_{7}$, then $\mathcal{SM}_{6} = [v_{6}, w_{0}^{J_{6}}]^{J_{6}}$. Therefore, we obtain $\text{dim}\, E_{\overline{v_{6}}} (\Lambda_{6}) = 43$.

\begin{table}[htbp]
\centering
\begin{tabular}{c|c|c|c|c|c}
$\# \mathcal{SM}_{i}$ & $\text{E}_{6}$ & $\text{E}_{7}$ & $\text{E}_{8}$ & $\text{F}_{4}$ & $\text{G}_{2}$ \\ \hline \hline
$i=1$ & $16$ & $35$ & $71$ & \bf{--} & $1$ \\ \hline
$i=2$ & $4$ & $5$ & $6$ & \bf{--} & \bf{--} \\ \hline
$i=3$ & $1$ & $1$ & $1$ & $1$ & \bf{--} \\ \hline
$i=4$ & $4$ & $5$ & $6$ & $6$ & \bf{--} \\ \hline
$i=5$ & $16$ & $11$ & $16$ & \bf{--} & \bf{--} \\ \hline
$i=6$ & $12$ & $43$ & $27$ & \bf{--} & \bf{--} \\ \hline
$i=7$ & \bf{--} & $20$ & $105$ & \bf{--} & \bf{--} \\ \hline
$i=8$ & \bf{--} & \bf{--} & $30$ & \bf{--} & \bf{--}
\end{tabular}
\caption{The number of strong minuscule elements in $\mathcal{SM}_{i}$.}
\end{table}

\end{document}